\documentclass[12pt]{amsart}
\usepackage{amssymb}
\usepackage[all]{xy}

\newtheorem{theorem}{Theorem}[section]
\newtheorem{lemma}[theorem]{Lemma}
\newtheorem{prop}[theorem]{Proposition}
\newtheorem{corol}[theorem]{Corollary}
\theoremstyle{definition}
\newtheorem*{defn}{Definition}
\numberwithin{equation}{section}
\newtheorem{rmk}[theorem]{Remark}

\def\kA{\mathcal A}	\def\kB{\mathcal B}
\def\kC{\mathcal C}	\def\kI{\mathcal I}
\def\kZ{\mathcal Z}	\def\kR{\mathcal R}

\def\bK{\mathbf K}	\def\bA{\mathbf A}
\def\bD{\mathbf D}	\def\bF{\mathbf F}
\def\bR{\mathbf R}	\def\dT{\mathfrak T}
\def\gM{\mathfrak m}	\def\gK{\mathfrak k}

\def\sG{\mathsf G}	\def\sH{\mathsf H}
\def\sN{\mathsf N}	\def\kS{\mathcal S}

\def\iff{if and only if }	\def\ds{\displaystyle}

\def\set#1{\left\{\,#1\,\right\}}
\def\setsuch#1#2{\left\{\,#1\mid #2\,\right\}}

\def\al{\alpha}	\def\be{\beta}	
\def\si{\sigma}	\def\la{\lambda}
\def\ze{\zeta}	\def\eps{\varepsilon}
\def\io{\iota}	\def\Th{\Theta}
\def\ga{\gamma}

\def\bop{\bigoplus}	\def\ti{\tilde}
\def\ito{\stackrel\sim\to}	\def\dd{\partial}
\def\+{\oplus}	\def\xx{\times}
\def\*{\otimes}	\def\ol{\overline}

\def\add{\mathop\mathrm{add}\nolimits}
\def\mor{\mathop\mathrm{Mor}\nolimits}
\def\ob{\mathop\mathrm{Ob}\nolimits}
\def\rad{\mathop\mathrm{rad}\nolimits}
\def\chr{\mathop\mathrm{char}\nolimits}

\def\Md{\mbox{-}\mathsf{Mod}}
\def\el{\mathop\mathsf{El}\nolimits}
\def\dto{\dashrightarrow}	\def\Arr{\Rightarrow}

\def\Ker{\mathop\mathrm{Ker}\nolimits}
\def\cok{\mathop\mathrm{Cok}\nolimits}
\def\id{\mathop\mathrm{id}\nolimits}
\def\tr{\mathop\mathrm{tr}\nolimits}
\def\Hom{\mathop\mathrm{Hom}\nolimits}
\def\2{^{(2)}}	\def\ola{\bar\la}
\def\lsi{\bar\si}	
\def\lA{\ol\kA}

\def\pev#1{\noindent{\it Proof} is evident#1 \qed}
\def\LA{\forall}	\def\hG{\hat\sG}
\def\op{^{\mathrm{op}}}	\def\bup{\bigcup}
\def\sbe{\subseteq}	\def\spe{\supseteq}
\def\xarr{\xrightarrow}

\begin{document}
 \author{Yuriy A. Drozd} 
 \title{Group action on bimodule categories}
 \dedicatory{To the memolry of A.\,V.~Roiter}

 \address{Institute of Mathematics, National Academy of Sciences of Ukraine,
 Tereschenkivska 3, 01601 Kiev, Ukraine}
 \email{drozd@imath.kiev.ua}
 \thanks{This research was partially supported by the INTAS Grant no. 06-1000017-9093}

 \begin{abstract}
  We consider actions of groups on categories and bimodules, the related crossed group categories and bimodules,
 and prove for them analogues of the result know for representations of crossed group algebras and categories.
 \end{abstract}

 \maketitle

 Skew group algebras arise naturally in lots of questions. In particular, the properties
 of the categories of representations of skew group algebras and, more generally,
 skew group categories have been studied in \cite{rr,dof}. On the other hand, ``\emph{matrix
 problems},'' especially, \emph{bimodule categories} play now a crucial role in the theory of
 representations \cite{cb,d1}. The situation, when a group acts on a bimodule, thus also
 on the bimodule category is also rather typical. Therefore one needs to deal with \emph{skew
 bimodules} and their bimodule categories. In this paper we shall study skew bimodules
 and bimodule categories and prove for them some analogues of the results of \cite{rr,dof}. 

 In Section \ref{s1} we recall general notions related to bimodule categories. In Section \ref{s2}
 we consider actions of groups on bimodule and bimodule categories and the arising functors.
 The main results are those of Section \ref{s3}, where we define \emph{separable actions}
 and prove that in the separable case the bimodule category of the skew bimodule is 
 equivalent to the skew category of the original one. We also consider specially the case 
 of the abelian groups, since in this case the original category can be restored from the
 skew one using the group of characters. Section \ref{s4} is devoted to the decomposition
 of objects in skew group categories, especially, to the number of non-isomorphic
 direct summands in such decompositions. We also consider the \emph{radical} and
 \emph{almost split morphisms} in the skew group categories (under the separability
 condition).

 \section{Bimodule categories}
 \label{s1}

 We recall the main definitions related to bimodule categories
 \cite{cb,d1}. We fix a commutative ring $\bK$. All categories that we consider
 are supposed to be \emph{$\bK$-categories}, which means that all sets
 of morphisms are $\bK$-modules, while the multiplication is $\bK$-bilinear.
 We denote the set of morphisms from an object $X$ to an object $Y$ in a
 category $\kA$ by $\kA(X,Y)$. A \emph{module} (more precise, a
 \emph{left module}) over a category $\kA$, or a $\kA$\emph{-module} is,
 by definition, a $\bK$-linear functor $M:\kA\to\bK\Md$, where $\bK\Md$
 denotes the category of $\bK$-modules. If $M$ is such a module,
 $x\in M(X)$ and $a\in\kA(X,Y)$, we write, as usually, $ax$ instead of $M(a)(x)$.
 Such modules have all usual properties of modules over rings. The category
 of all $\kA$-modules is denoted by $\kA\Md$. A \emph{bimodule} over a category
 $\kA$, or an $\kA$-bimodule, is, by definition, a $\bK$-bilinear functor
 $\kB:\kA\op\xx\kA\to\bK\Md$, where $\kA\op$ is the opposite category to
 $\kA$. If $x\in\kB(X,Y)$, $a:X'\to X$ (i.e. $a:X\to X'$ in $\kA\op$), $b:Y\to Y'$, 
 we write $bxa$ instead of $\kB(a,b)(x)$ (this element belongs to $\kB(X',Y')$).
 In particular, $xa$ and $bx$ denote, respectively, $\kB(a,1_Y)(x)$ and
 $\kB(1_X,b)(x)$. If a bimodule $\kB$ is fixed, we often write $x:X\dto Y$ 
 instead of $x\in\kB(X,Y)$.

 A category $\kA$ is called \emph{fully additive} if it is additive (i.e. has direct sums
 $X\+Y$ of any pair of objects $X,Y$ and a zero object $0$) and every idempotent
 endomorphism $e\in\kA(X,X)$ \emph{splits}, i.e. there is an object $Y$ and a pair
 of morphisms $\io:Y\to X$ and $\pi:X\to Y$ such that $\pi\io=1_Y$ and $\io\pi=e$. 
 Choosing an object $Y'$ and morphisms $\io':Y'\to X$ and $\pi':X\to Y'$ such that
 $\pi'\io'=1_{Y'}$ and $\io'\pi'=1-e$, we present $X$ as a direct sum $Y\+Y'$, where
 $\io$ and $\io'$ are canonical embeddings, while $\pi$ and $\pi'$ are canonical
 projections. For every $\bK$-category $\kA$ there is the smallest fully additive
 category $\add\kA$ containing $\kA$. This category is unique (up to equivalence).
 It can be identified either with the category of matrix idempotents over $\kA$ or
 with the category of finitely generated projective $\kA$-modules \cite{gr}. We
 call it the \emph{additive hull} of $\kA$. Each $\kA$-module $M$ (bimodule $\kB$)
 extends uniquely (up to isomorphism) to a module (bimodule) over the category
 $\add\kA$, which we also denote by $M$ (respectively, by $\kB$)

 If $\kB$ is an $\kA$-bimodule, a \emph{differentiation} from
 $\kA$ to $\kB$ is, by definition, a set of $\bK$-linear maps
 \[
  \dd=\setsuch{\dd(X,Y):\kA(X,Y)\to\kB(X,Y)}{X,Y\in\ob\kA},  
 \]
  satisfying the \emph{Leibniz rule}:
 \[
 \dd(ab)=(\dd a)b+a(\dd b)
 \]
 for any morphisms $a,b$ such that the product $ab$ is defined. It implies, in
 particular, that $\dd 1_X=0$ for any object $X$. Again, such a differentiation
 extends to the additive hull of $\kA$ and we denote this extension by the same
 letter $\dd$. A triple $\dT=(\kA,\kB,\dd)$, where $\kA$ is a category, $\kB$ is a
 $\kA$-bimodule and $\dd$ is a differentiation from $\kA$ to $\kB$, is called
 a \emph{bimodule triple}. If $\dT'=(\kA',\kB',\dd')$ is another bimodule triple,
 a \emph{bifunctor} from $\dT$ to $\dT'$ is defined as a pair $F=(F_0,F_1)$, 
 where $F_0:\kA\to\kA'$ is a functor, $F_1:\kB\to\kB'(F_0)$ is a homomorphism
 of $\kA$-bimodule, where $\kB'(F_0)$ is the $\kA$-bimodule obtained from $\kB'$
 by the transfer along $F_0$ (i.e. $F_1(x):F_0(X)\dto F_0(Y)$ if $x:X\dto Y$, and
 $F_1(bxa)=F_0(b)F_1(x)F_0(a)\hskip1pt$), such that
 $F_1(\dd a)=\dd'(F_0(a))$ for all $a\in\mor\kA$. As a rule, we write
 $F(a)$ and $F(x)$ instead of $F_0(a)$ and $F_1(x)$. 	

 Let $F=(F_0,F_1)$ and $G=(G_0,G_1)$ be two bifunctors from a triple
 $\dT=(\kA,\kB,\dd)$ to another triple $\dT'=(\kA',\kB',\dd')$. A \emph{morphism
 of bifunctors} $\phi:F\to G$ is defined as a morphism of functors $\phi:F_0\to G_0$
 such that
 \begin{align*}
 & \phi(Y)F_1(x)=G_1(x)\phi(X)\ \text{ for each }\ x\in\kB(X,Y),\\
 & \dd'\phi(X)=0\ \text{ for each }\ X\in\ob\kA.
 \end{align*}
 If $\phi$ is an isomorphism of functors, the inverse morphism
 is obviously a morphism of bifunctors too. Then we call $\phi$ an
 \emph{isomorphism of bifunctors} and write $\phi:F\ito G$. If such an
 isomorphism exists, we say that the bifunctors $F$ are $G$
 \emph{isomorphic} and write $F\simeq G$.
 
 We call a bifunctor $F:\dT\to\dT'$ an \emph{equivalence of bimodule triples}
 if there is such a bifunctor $G:\dT'\to\dT$ that $FG\simeq\id_{\dT'}$ and 
 $GF\simeq\id_{\dT}$, where $\id_\dT$ denotes the identity bifunctor $\dT\to\dT$.
 If such a bifunctor exists, we call the triples $\dT$ and $\dT'$
 \emph{equivalent} and write $\dT\simeq\dT'$.

 \begin{lemma}\label{l11}
 A bifunctor $F=(F_0,F_1)$ is an equivalence of bimodule triples \iff the following
 conditions hold:
 \begin{enumerate}

\item   The functor $F_0$ is \emph{fully faithful}, i.e. all induced maps
 $\kA(X,Y)\to\kA'(F_0X,F_0Y)$ are bijective.
 \item   This functor is also $\dd$-\emph{dense}, i.e. for every object
 $X'$ of the category $\kA'$ there are an object $X\in\ob\kA$ and an isomorphism
 $\al:X'\to F_0X$ such that $\dd\al=0$ .
 \item    The map $F_1(X,Y):\kB(X,Y)\to\kB'(F_0X,F_0Y)$ is bijective for any
 $X,Y\in\ob\kA$.
\end{enumerate}
 Moreover, if these conditions hold, there is a bifunctor $G:\dT'\to\dT$ and
 an isomorphism $\la:\id_{\dT'}\to FG$ such that  $GF=\id_{\dT}$ and $\la(FX)=1_{FX}$
 for all $X\in\ob\kA$.
 \end{lemma}
\begin{proof}
 The necessity of these conditions is evident, so we prove their sufficiency. 
 Suppose that these conditions hold. For each object $X'\in\kA'$ choose an object
 $X$ and an isomorphism $\al:X'\to F_0X$ such that $\dd a=0$, always setting $\al=1_{X'}$
 for $X'=F_0X$. Set $G_0X'=X$ and $\la(X')=\al$. For each morphism $a:X'\to Y'$ set
 $G_0a=F_0^{-1}(X,Y)(\la(Y')a\la^{-1}(X'))$, where $X=G_0X',\,Y=G_0Y'$
 (then $\la(X'):X'\ito F_0X,\ \la(Y'):Y'\ito F_0Y$). Obviously, the set
 $\{\la(X')\}$ defines an  isomorphism of functors $\la:\id\to F_0G_0$. We also define
 a homomorphism of bimodules $G_1:\kB'\to\kB(G_0)$ setting
 $G_1(x)=F_1(X,Y)^{-1}(\la(Y')x\la^{-1}(X'))$ if $x:X'\dto Y',\ 
 X=G_0X',\ Y=G_0Y'$. Then $G=(G_0,G_1)$ is a bifunctor
 $\dT'\to\dT$ and $\la$ is an isomorphism of bifunctors $\id_{\dT'}\to FG$.
 Moreover, by this construction, $GF=\id_\dT$ and $\la(FX)=1_{FX}$ for all $X$.
 \end{proof}

 Every bimodule triple $\dT=(\kA,\kB,\dd)$ gives rise to the  \emph{bimodule
 category} (or the \emph{category of representations}, or the \emph{category of elements})
 of this triple \cite{cb}. The \emph{objects} of this category are elements $\bup_X\kB(X,X)$, where
 $X$ runs through objects of the category $\add\kA$. \emph{Morphisms} from an object $x:X\dto X$ 
 to an object $y:Y\dto Y$ are such morphisms $a:X\to Y$ that $ax=ya+\dd(a)$ in $\kB(X,Y)$. 
 It is easy to see that these definitions really define a fully additive $\bK$-category
 $\el(\dT)$. The set of morphisms $x\to y$ in this category is denoted by $\Hom_\dT(x,y)$.
 If $\dd=0$, we write $\el(\kA,\kB)$ or even $\el(\kB)$ instead of $\el(\kA,\kB,\dd)$.
 Each bifunctor between bimodule triples $F:\dT\to\dT'$ gives rise to
 a functor $F_*:\el(\dT)\to\el(\dT')$, which maps an object $x$ to the object
 $F_1(x)$ and a morphism $a:x\to y$ to the morphism $F_0(a):F_1(x)\to F_1(y)$.
 As well, each morphism of bifunctors $\phi:F\to G$ induces a morphism of
 functors $\phi_*:F_*\to G_*$, which correlate an object $x\in\kB(X,X)$
 with the morphism $\phi(X)$ considered as a morphism $F(x)\to G(x)$. Obviously,
 if $\phi$ is an isomorphism of bifunctors, $\phi_*$ is an isomorphism of functors.
 Especially, if $F$ is an equivalence of bimodule triples, the functor
 $F_*$ is an equivalence of their bimodule categories.	 

 If $\kB=\kA$ and $\dd=0$, we say that the bimodule triple $\dT=(\kA,\kA,0)$ is the
 \emph{principle triple} for the category $\kA$. Obviously, a bifunctor between
 principle triples is just a functor between the corresponding categories
 and a morphism of such bifunctors is just a morphism of functors.
 The bimodule category of the principle triple for a category $\kA$ is denoted by $\el(\kA)$.

 If $\kA$ and $\kA'$ are two categories, one can consider $\kA\mbox{-}\kA'$\emph{-bimodules},
 i.e. bilinear functors $\kB:\kA\op\xx\kA'\to\bK\Md$. Actually, any such bimodule can be
 identified with a $\kA\xx\kA'$-bimodule $\ti\kB$ with $\ti\kB((X,X'),(Y,Y'))=
 \kB(X,Y')$ and $(a,a')x(b,b')=axb'$. Such bimodules are called \emph{bipartite}. 
 In particular, every $\kA$-bimodule $\kB$ defines a bipartite
 $\kA\mbox{-}\kA$-bimodule, which we denote by $\kB\2$ and call the
 \emph{double} of the $\kA$-bimodule $\kB$. Certainly, bimodules $\kB$ and $\kB\2$ are
 quite different and they define different bimodule categories. If $\kB=\kA$ the category
 $\el(\kA\2)$ coincides with the \emph{category of morphisms} of the additive hull
 $\add\kA$. 

 Further on we often identify the categories $\kA$ and $\add\kA$ and say
 ''an object (morphism) of $\kA$'' instead of ``an object (morphism) of $\add\kA$.''
 We hope that this petty ambiguity will not embarrass the reader.

 \section{Group actions}
  \label{s2}
 
 Let $\dT=(\kA,\kB,\dd)$ be a bimodule triple and $\sG$ be a group. One says that the group
 $\sG$ \emph{acts on the triple} $\dT$ if a bifunctor $T_\si:\dT\to\dT$ is defined for each
 element $\si\in G$ so that $T_1=\id_\dT$ and $T_{\si\tau}\simeq T_\si T_\tau$ for any
 $\si,\tau\in\sG$. It implies, in particular, that all $T_\si$ are equivalences.
 Further on we write $X^\si$ instead of $T_\si(X)$. We only note that according to this
 notation $X^{\si\tau}\simeq (X^\tau)^\si$. A \emph{system of factors} $\la$ for such an action 
 is defined as a set of isomorphisms of bifunctors $\la_{\si,\tau}: T_{\si\tau}\ito T_\si T_\tau$,
 which satisfy the relations: 
 \begin{equation}\label{e21}
    \la^\rho_{\si,\tau}\la_{\rho,\si\tau}=\la_{\rho,\si}\la_{\rho\si,\tau}
 \end{equation}
 for any triple of elements $\rho,\si,\tau\in\sG$, and $\la_{\si,1}=\la_{1,\si}=1$
 for any $\si\in\sG$. We omit the arguments (objects of $\kA$) in these formulae (and later
 on in analogous cases), since their values can easily be restored. Since $\la_{\si,\tau}$ is
 a morphism of bifunctors, one has $\la_{\si,\tau}:X^{\si\tau}\to(X^\tau)^\si$ and
 \begin{equation}\label{e22}
 \la_{\si,\tau} x^{\si\tau}=(x^\tau)^\si \la_{\si,\tau}
 \end{equation}
 for every morphism from $\kA$ and every element from $\kB$, and also
 $\dd\la_{\si,\tau}=0$ for all $\si,\tau$. Note also that the relations
 \eqref{e21} and \eqref{e22} imply, in particular, that
 \[
  \la^{\si}_{\si^{-1},\si}=\la_{\si,\si^{-1}}\ \text{ and }\ 
 \la_{\si,\si^{-1}}x=(x^{\si^{-1}})^\si\la_{\si,\si^{-1}}.  
 \]

 Given an action $T=\set{T_\si}$ of a group $\sG$ on a bimodule triple $\dT=(\kA,\kB,\dd)$
 and a system of factors $\la$ for this action, we define the \emph{crossed group triple}
 $\dT\sG=\dT(\sG,T,\la)$. Namely, we consider the \emph{crossed group category}
 $\kA\sG=\kA(\sG,T,\la)$ \cite{rr,dof}. Its objects coincide with those of $\kA$, but 
 morphisms $X\to Y$  in the category $\kA\sG$ are defined as formal (finite) linear
 combinations $\sum_{\si\in\sG}a_\si[\si]$, where $a_\si\in\kA(X^\si,Y)$, and the
 multiplication of such morphisms is defined by bilinearity and the rule
 \begin{equation}\label{e23}
  a_\si[\si] b_\tau[\tau] =a_\si b^\si_\tau\la_{\si,\tau}[\si\tau].
 \end{equation}
  The condition \eqref{e21} for a system of factors is equivalent to the associativity
 of this multiplication. The $\kA\sG$-bimodule $\kB\sG=\kB(\sG,T,\la)$ is constructed
 in an analogous way: elements of $\kB\sG(X,Y)$ are formal (finite) linear combinations
 $\sum_{\si\in\sG}x_\si[\si]$, where $x_\si\in\kB(X^\si,Y)$, and their products with
 morphisms from $\kA\sG$ are defined by the same formula \eqref{e23}, with the only
 difference that one of the elements $a_\si, b_\tau$ is a morphism from $\kA$, while
 the second one is an element from $\kB$. The differentiation $\dd$ extends to $\kA\sG$
 if we set $\dd(\sum_\si a_\si[\si])=\sum_\si \dd a_\si[\si]$. We identify every morphism
 $a\in\kA(X,Y)$ with the morphism $a[1]\in\kA\sG(X,Y)$ and every element
 $x\in\kB(X,Y)$ with the element $x[1]\in\kB\sG(X,Y)$
 getting the embedding bifunctor
 $\dT\to\dT\sG$.
  
 An action $T$ of a group $\sG$ on a bimodule triple $\dT$ induces its action $T_*$ on
 the bimodule category $\el(\dT)$: an element $\si\in\sG$ defines the functor
 $(T_\si)_*:x\mapsto x^\si$. Moreover, if $\la$ is a system of factors for the action $T$, 
 it induces the system of factors $\la_*$ for the action $T_*$: one has to set
 $(\la_*)_{\si,\tau}(x)= \la_{\si,\tau}(X)$ if $x\in\kB(X,X)$. Thus the crossed group category
 $\el(\dT)\sG=\el(\dT)(\sG,T_*,\la_*)$ is defined, as well as the embedding $\el(\dT)\to\el(\dT)\sG$.
 One can also define the natural functor $\Phi:\el(\dT)\sG\to\el(\dT\sG)$ as follows.
 For an object $x\in\kB(X,X)$, set $\Phi(x)=x[1]\in\kB\sG(X,X)$. Let $\al=\sum_\si a_\si[\si]$ be
 a morphism from $x$ to $y\in\kB(Y,Y)$ in the category $\el(\dT)\sG$. It means that
 $a_\si:x^\si\to y$ in the category $\el(\dT)$, i.e. $a_\si\in\kA(X^\si,Y)$ and
 $a_\si x^\si=ya_\si+\dd a_\si$. Then one can consider $\al$ as a morphism $X\to Y$
 in the category $\kA\sG(X,Y)$, and $\al x[1]=\sum_\si a_\si[\si]x[1] 
 =\sum_\si a_\si x^\si[\si]=\sum_\si (ya_\si+\dd a_\si)[\si]=y[1]\al+\dd\al$,
 so $\al$ is a morphism $x[1]\to y[1]$ in the category $\el(\dT\sG)$ and one can
 set $\Phi(\al)=\al$.

 \begin{prop}\label{p21}
 The functor $\Phi$ is fully faithful, i.e. for any objects $x,y$ from $\el(\dT)\sG$ it
 induces the bijective map ${\Hom_\dT}\sG(x,y)\to\Hom_{\dT\sG}(x,y)$, where
 ${\Hom_\dT}\sG$ denotes the morphisms in the category $\el(\dT)\sG$.
 \end{prop}
 \begin{proof}
 Obviously, this map is injective. Let $\al=\sum_\si a_\si[\si]:x[1]\to y[1]$, i.e.
 $\al x[1]=\sum_\si a_\si x^\si[\si]=y[1]\al+\dd\al=\sum_\si (ya_\si+\dd a_\si)[\si]$. Then
 $a_\si x^\si=ya_\si+\dd a_\si$ for all $\si$, so $a_\si:x^\si\to y$ in the category $\el(\dT)$,
 thus $\al:x\to y$ in the category $\el(\dT)\sG$. Therefore, this map is also surjective. 
 \end{proof}
 
 If the group $\sG$ is finite, one can also construct a functor $\Psi:\el(\dT\sG)\to \el(\dT)$.
 For every object $X\in\ob\kA$, set $\ti X=\bop_{\si\in\sG}X^\si$ and for every element
 $\xi=\sum_\si x_\si[\si]\in\kB\sG(X,X)$, where $x_\si:X^\si\dto X$, denote by
 $\ti\xi$ the element from $\kB(\ti X,\ti X)=\bop_{\si,\tau}\kB(X^\tau,X^\si)$ such that its
 component $\ti\xi_{\si,\tau}\in\kB(X^\tau,X^\si)$ equals
 $x^\si_{\si^{-1}\tau}\la_{\si,\si^{-1}\tau}$. Note that
 $x_{\si^{-1}\tau}:X^{\si^{-1}\tau}\dto Y$, hence
 $x^\si_{\si^{-1}\tau}: (X^{\si^{-1}\tau})^\si\dto Y^\si$, thus 
 $x^\si_{\si^{-1}\tau}\la_{\si,\si^{-1}\tau}:X^\tau\dto Y^\si$ indeed. Let
 $\eta=\sum_\si y_\si[\si]\in\kB\sG(Y,Y)$, where $y_\si\in\kB(Y^\si,Y)$ and
 $\al=\sum_\si a_\si[\si]$ be a morphism from $\xi$ to $\eta$, where
 $a_\si\in\kA(X^\si,Y)$. Since
 \begin{align*}
   \al\xi &= \sum_\rho\sum_\si a_\rho[\rho] x_\si[\si] 
 =\sum_\rho\sum_\si a_\rho x^\rho_\si \la_{\rho,\si} [\rho\si] =\\
  &= \sum_\tau\big(\sum_\rho a_\rho x^\rho_{\rho^{-1}\tau}\la_{\rho,\rho^{-1}\tau}\big)[\tau],  \\
 \intertext{and}
   \eta\al &= \sum_\rho\sum_\si y_\rho[\rho] a_\si[\si] 
 =\sum_\rho\sum_\si y_\rho a^\rho_\si \la_{\rho,\si} [\rho\si]=\\
 & = \sum_\tau\big(\sum_\rho y_\rho a^\rho_{\rho^{-1}\tau}\la_{\rho,\rho^{-1}\tau}\big)[\tau] ,
 \end{align*}
 it means that, for each $\tau$,
 \begin{equation}\label{e24}
  \sum_\rho a_\rho x^\rho_{\rho^{-1}\tau}\la_{\rho,\rho^{-1}\tau}
  = \sum_\rho y_\rho a^\rho_{\rho^{-1}\tau}\la_{\rho,\rho^{-1}\tau}+\dd a_\tau . 
 \end{equation}
 Consider the morphism $\ti\al:\ti X\to \ti Y$ such that
 \[
 \ti\al_{\si,\tau}=a^\si_{\si^{-1}\tau}\la_{\si,\si^{-1}\tau}:X^\tau\to Y^\si. 
 \]
 Then the $(\si,\tau)$-component of the product $\ti\al\ti\xi$ equals
\[
 \mathrm I=\sum_\rho a^\si_{\si^{-1}\rho}\la_{\si,\si^{-1}\rho}x^\rho_{\rho^{-1}\tau}\la_{\rho,\rho^{-1}\tau}=
 \sum_\rho a^\si_{\si^{-1}\rho} (x^{\si^{-1}\rho}_{\rho^{-1}\tau}\big)^\si
 \la_{\si,\si^{-1}\rho}\la_{\rho,\rho^{-1}\tau},
\]
 while the $(\si,\tau)$-component of the product $\ti\eta\ti\al$ equals
\[
 \mathrm{II}= \sum_\rho y^\si_{\si^{-1}\rho}\la_{\si,\si^{-1}\rho}a^\rho_{\rho^{-1}\tau}\la_{\rho,\rho^{-1}\tau}=
 \sum_\rho y^\si_{\si^{-1}\rho} (a^{\si^{-1}\rho}_{\rho^{-1}\tau}\big)^\si
 \la_{\si,\si^{-1}\rho}\la_{\rho,\rho^{-1}\tau}.
 \]
 (In both cases we used the relation \eqref{e22} replacing $\tau$ by $\si^{-1}\rho$).
 Since, by the condition \eqref{e21} for the system of factors,
 \[
  \la_{\si,\si^{-1}\rho}\la_{\rho,\rho^{-1}\tau}=
   \la^\si_{\si^{-1}\rho,\rho^{-1}\tau}\la_{\si,\si^{-1}\tau},\ \text{ and }\ \dd\la_{\si,\si^{-1}\tau}=0 , 
 \]
 we get from the relation \eqref{e24} that $\mathrm I=\mathrm{II}+\dd\ti\al_{\si,\tau}$
 (we just replace $\rho$ by $\si^{-1}\rho$, $\tau$ by $\si^{-1}\tau$, then apply the functor
 $T_\si$ to both sides). Therefore, $\ti\al$ is a morphism $\ti\xi\to\ti\eta$ and one can define
 the functor $\Psi$ setting $\Psi(\xi)=\ti\xi$ and $\Psi(\al)=\ti\al$.

 \begin{prop}\label{p22}
 The functors $\Phi$ and $\Psi$ form an \emph{adjoint pair}, i.e. there is a natural isomorphism
 ${\Hom_\dT}\sG(\Phi x,\eta)\simeq\Hom_\dT(x,\Psi\eta)$ for each objects $x\in\el(\dT)$ and
 $\eta\in\el(\dT\sG)$.
 \end{prop}
 \begin{proof}
  Let $x\in\kB(X,X),\ \eta\in\kB\sG(Y,Y),\ \eta=\sum_\si y_\si[\si]$, where
 $y:Y^\si\dto Y$, and $\al:\Phi (x)=x[1]\to\eta$ in the category $\el(\dT\sG)$. By definition,
 $\al=\sum_\si a_\si[\si]$, where $a_\si:X^\si\to Y$, and
 \[
 \al x[1]=\sum_\si a_\si x^\si[\si]= \eta \al +\dd\al=
  \sum_\si\big(\sum_\rho y_\rho a^\rho_{\rho^{-1}\si}\la_{\rho,\rho^{-1}\si}+\dd a_\si\big)[\si], 
 \]
 i.e.
 \begin{equation}\label{e25}
 a_\si x^\si=\sum_\rho y_\rho a^\rho_{\rho^{-1}\si}\la_{\rho,\rho^{-1}\si}+\dd a_\si
 \end{equation}
 for every $\si$. Consider the morphism $f(\al)=\be:X^\tau\to\ti Y=\bop_\si Y^\si$ such that
 its component $\be_\si:X\to Y^\si$ equals $a^\si_{\si^{-1}}\la_{\si,\si^{-1}}$. Compute
 the $\si$-components of the products $\be x$ and $\ti\eta\be$, where $\ti\eta=\Psi\eta$. 
 They equal, respectively,
 \begin{align*}
  & \be_\si x^\tau=a^\si_{\si^{-1}}\la_{\si,\si^{-1}}x=
  a^\si_{\si^{-1}}(x^{\si^{-1}})^\si\la_{\si,\si^{-1}}\\
 \intertext{and}
 &\sum_\rho y^\si_{\si^{-1}\rho}\la_{\si,\si^{-1}\rho}a^\rho_{\rho^{-1}}
 \la_{\rho,\rho^{-1}}=
 \sum_\rho y^\si_{\si^{-1}\rho}(a^{\si^{-1}\rho}_{\rho^{-1}})^\si
 \la_{\si,\si^{-1}\rho}\la_{\rho,\rho^{-1}}=\\
 &\ =\sum_\rho y^\si_{\si^{-1}\rho}(a^{\si^{-1}\rho}_{\rho^{-1}})^\si
 \la^\si_{\si^{-1}\rho,\rho^{-1}}\la_{\si,\si^{-1}}.
 \end{align*}
 The relation \eqref{e25}, where $\si$ is replaced by $\si^{-1}$
 and $\rho$ by $\si^{-1}\rho$, these two expressions differ exactly by
 $\dd\be_\si=\dd a^\si_{\si^{-1}}\la_{\si,\si^{-1}}$, hence $\be=f(\al)$ 
 is a morphism $x\to \ti\eta$ in the category $\el(\dT)$. Obviously, if $\al\ne\al'$,
 then $f(\al)\ne f(\al')$ as well. Moreover, one easily checks that the correspondence
 $\al\mapsto f(\al)$ is functorial in $x$ and $\eta$, i.e.  $f(\al) b=f(\al\Phi b)$ and
 $f(\ga\al)=(\Psi\ga)f(\al)$ for any morphisms $b:x'\to x$ and $\ga:\eta\to\eta'$.

 On the contrary, let $\be:x\to \ti\eta$ be a morphism in the category $\el(\dT)$.
 Denote by $\be_\si:X\to Y^\si$ the corresponding component of $\be$ and
 consider the morphism $\al=\sum_\si a_\si[\si]:X\to Y$ in the category $\kA\sG$,
 where $a_\si=\la^{-1}_{\si,\si^{-1}}\be^\si_{\si^{-1}}:X^\si\to Y$. Comparing
 the $\si$-components in the equality $\be x=\ti\eta\be$, we get
 \begin{equation}\label{e26}
   \be_\si x=\sum_\rho y^\si_{\si^{-1}\rho}\la_{\si,\si^{-1}\rho}\be_\rho+\dd \be_\si.
 \end{equation}
 The coefficients near $[\si]$ in the products $\al(\Phi x)=\al x[1]$ and $\eta\al$
 equal, respectively, 
 \begin{align*}
    a_\si x^\si&=\la^{-1}_{\si,\si^{-1}}\be^\si_{\si^{-1}}x^\si \\
 \intertext{and}
  \sum_\rho y_\rho a^\rho\la_{\rho^{-1}\si}&=
 \sum_\rho y_\rho (\la^{-1}_{\rho^{-1}\si,\si^{-1}\rho})^\rho
 \be^{\rho^{-1}\si}_{\si^{-1}\rho}\la_{\rho,\rho^{-1}\si} =\\
 &=\sum_\rho y_\rho (\la^{-1}_{\rho^{-1}\si,\si^{-1}\rho})^\rho
 \la_{\rho,\rho^{-1}\si}\be^\si_{\si^{-1}\rho} . 
 \end{align*}
 The relation \eqref{e26}, with $\si$ replaced by $\si^{-1}$, implies that
 \begin{align*}
  a_\si x^\si -\dd a_\si=& 
  \sum_\rho \la^{-1}_{\si,\si^{-1}}(y^{\si^{-1}}_{\si\rho})^\si
  \la^\si_{\si^{-1},\rho}\be^\si_{\si^{-1}\rho} =\\
  =&\sum_\rho y_{\si\rho}\la^{-1}_{\si,\si^{-1}} 
  \la^\si_{\si^{-1},\rho}\be^\si_{\si^{-1}\rho} =
 \sum_\rho y_\rho\la^{-1}_{\si,\si^{-1}}
 \la^\si_{\si^{-1},\si\rho}\be^\si_{\rho} = \\
  =& \sum_\rho y_\rho\la^{-1}_{\si,\si^{-1}\rho}\be^\si_{\si^{-1}\rho}=
   \sum_\rho y_\rho (\la^{-1}_{\rho^{-1}\si,\si^{-1}\rho})^\rho
 \la_{\rho,\rho^{-1}\si}\be^\si_{\si^{-1}\rho} .
 \end{align*}
 (Passing from the second row to the third, we used the relation \eqref{e21} 
 for the triple $\si,\si^{-1},\rho$, while in the third row we used the same relation
 for the triple $\rho,\rho^{-1}\si,\si^{-1}\rho$.) Therefore,
 $\al x[1]=\eta\al+\dd\al$, thus $\al$ is a morphism $\Phi x\to \eta$. Moreover, the
 $\si$-component of $f(\al)$ equals 
 \[
 a^\si_{\si^{-1}}\la_{\si,\si^{-1}} =
 (\la^{-1}_{\si^{-1},\si})^\si(\be^{\si^{-1}}_\si)^\si \la_{\si,\si^{-1}}=
 (\la^{-1}_{\si^{-1},\si})^\si\la_{\si,\si^{-1}}\be_\si=\be_\si.
 \]
 Hence $f(\al)=\be$ and the map $\al\mapsto f(\al)$ is bijective.
 \end{proof}

 \section{Separable actions}
 \label{s3}

 We call the \emph{center} $\kZ(\dT)$ of a bimodule triple $\dT=(\kA,\kB,\dd)$ the
 endomorphism ring of the identity bifunctor $\id_\dT$. In other words, the elements
 of this center are the sets of morphisms 
 \[
  \al=\setsuch{\al_X:X\to X}{X\in\ob\kA},
 \]
 such that $\al_Ya=a\al_X$ for every morphism $a:X\to Y$, $\al_Yx=x\al_X$
 for every element $x:X\dto Y$ and $\dd\al_X=0$ for all $X$. In particular,
 the element $\al_X$ belongs to the center of the algebra $\kA(X,X)$. One
 easily sees that if $\al=\set{\al_X}$ and $\be=\set{\be_X}$ are two such sets,
 then the sets $\al+\be=\set{\al_X+\be_X}$ and $\al\be=\set{\al_X\be_X}$
 also belong to $\kZ(\dT)$. Hence, this center is a ring (even a $\bK$-algebra),
 commutative, since $\al_X\be_X=\be_X\al_X$. If $F=(F_0,F_1)$ is an equivalence
 of bimodule triples $\dT\to\dT'=(\kA',\kB',\dd')$, it induces an isomorphism
 $F_\kZ:\kZ(\dT)\ito\kZ(\dT')$. Namely, for any $X'\in\ob\kA'$, choose an
 isomorphism $\la:X'\to F_0X$ for some $X\in\ob\kA$, and, for each element 
 $\al=\set{\al_X}\in\kZ(\dT)$, set $(F_\kZ\al)_{X'}=\la^{-1}(F_0\al_X)\la$. 
 Let $Y'$ be another object from $\kA$, $\mu:Y'\ito F_0Y$ and
 $(F_\kZ\al)_{Y'}=\mu^{-1}(F_0\al_Y)\mu$. If $a'\in\kA'(X',Y')$, the morphism
 $\mu a'\la^{-1}:F_0X\to F_0Y$ is of the form $F_0a$ for some $a:X\to Y$. 
 It gives
 \begin{equation}\label{e31}
 \begin{split}
  (F_\kZ\al)_{Y'}a'&=\mu^{-1}(F_0\al_Y)\mu\cdot\mu^{-1}(F_0a)\la= \\&=
  \mu^{-1}(F_0\al_Y)(F_0a)\la=\mu^{-1}(F_0(\al_Ya))\la= \\&=
 \mu^{-1}F_0(a\al_X)\la=\mu^{-1}(F_0a)(F_0\al_X)\la=\\
 & =a'\la^{-1}(F_0\al_X)\la =a'(F_\kZ\al)_{X'}.    
 \end{split}
 \end{equation}
 Especially, if $Y'=X'$ and $a'=1_{X'}$, we see that $F_\kZ(\al)_{X'}$ does
 not depend on the choice of $X$ and $\la$. Just in the same way one checks
 that $(F_\kZ\al)_{Y'}x'=x'(F_\kZ\al)_{X'}$ for every $x'\in\kB'(X',Y')$. Note
 that an isomorphism $\la$ can always be chosen such that $\dd\la=0$: 
 for instance, one can use the isomorphism of bifunctors
 $\phi:\id_{\dT'}\to FG$ for some bifunctor $G$ and set
 $X=G_0X',\ \la=\phi(X')$. Therefore $\dd'(F_\kZ\al)_{X'}=0$,
 so the set $F_\kZ\al=\set{(F_\kZ\al)_{X'}}$ belongs to $\kZ(\dT')$.
 Obviously, $F_\kZ(\al+\be)=F_\kZ\al+F_\kZ\be$ and $F_\kZ(\al\be)=
 (F_\kZ\al)(F_\kZ\be)$, and if $F':\dT'\to\dT''$ is another equivalence,
 then $(F'F)_\kZ=F'_\kZ F_\kZ$. Moreover, similarly to the equalities \eqref{e31},
 one easily verifies that if $F\simeq F'$, then $F_\kZ=F'_\kZ$. In particular,
 if $G:\dT'\to\dT$ is such a bifunctor that $FG\simeq\id_{\dT'}$
 and $GF\simeq\id_\dT$, then $G_\kZ=F_\kZ^{-1}$, thus $F_\kZ$ is an
 isomorphism. 

 These considerations imply that every action $T$ of a group $\sG$ on a triple $\dT$
 induces an action of the same group on the center of this triple with the trivial system
 of factors: if $\la$ is a system of factors for the action $T$, then
 $(\al^\si)_X=\la_{\si,\si^{-1}}^{-1} \al^\si_{X^{\si^{-1}}} \la_{\si,\si^{-1}}$
 for every $\al\in\kZ(\dT)$. Especially, if the group $\sG$ is finite, for any
 element $\al$ from $\kZ(\dT)$ its \emph{trace} is defined as
 $\tr\al=\tr_\sG\al=\sum_\si\al^\si$, i.e. $(\tr\al)_X=\sum_\si \la_{\si,\si^{-1}}^{-1}
 \al^\si_{X^{\si^{-1}}} \la_{\si,\si^{-1}}$. Obviously, the center of the triple $\dT\sG$
 is a subalgebra of the center of $\dT$.

 \begin{prop}\label{p31}
 The center $\kZ(\dT\sG)$ coincides with the subalgebra $\kZ(\dT)^\sG$ of elements
 of the center $\kZ(\dT)$ that are invariant under the action of $\sG$. In particular, if
 this group is finite, the trace of each element $\al\in\kZ(\dT)$ belongs to $\kZ(\dT\sG)$.
 \end{prop}
 \begin{proof}
 Let $\al=\set{\al_X}$ be an element of the center $\kZ(\dT)$. Since
 $\al_Ya[\si]=a\al_{X^\si}[\si]$ and $a[\si]\al_X=a\al_X^\si[\si]$ for each
 morphism $a:X^\si\to Y$, this element belongs to the center of the triple
 $\dT\sG$ \iff $\al_{X^\si}=\al^\si_X$ for every $X$ and every $\si$. But then
 \[
   (\al^\si)_X=\la^{-1}_{\si,\si^{-1}}\al_{X^{\si^{-1}}}^\si\la_{\si,\si^{-1}} =
 \la^{-1}_{\si,\si^{-1}}(\al^{\si^{-1}}_X)^\si\la_{\si,\si^{-1}}=\al_X,
 \]
 so $\al$ is invariant under the action of $\sG$. Just in the same way one verifies that
 every invariant element from $\kZ(\dT)$ belongs to $\kZ(\dT\sG)$. The last statement
 follows from the fact that $\tr\al$ is always invariant under the action of the group.
 \end{proof}
 
 \begin{defn}\label{d31}
 We call an action of a finite group $\sG$ on a bimodule triple $\dT$ 
 \emph{separable}, if there is an element of the center $\al\in\kZ(\dT)$
 such that $\tr\al=1$.  
 \end{defn}

 Certainly, it is enough $\tr\al$ to be invertible. For instance, if the order of the
 group $\sG$ is invertible in the ring $\bK$, any action of this group is separable.
 Another important case is when the center of the triple $\dT$ contains a subring
 $\bR$ such that it is $\sG$-invariant, the group $\sG$ acts \emph{effectively}
 (i.e. for any $\si\ne1$ there is $r\in\bR$ such that $r^\si\ne r$) and $\bR$ is a
 \emph{separable extension} of its subring of invariants $\bR^\sG$ \cite{chr}. 
 If $\bR$ is a field and $\sG$ acts effectively on $\bR$, the last condition always
 holds. In general case it is necessary and sufficient
 that every element $\si\ne1$ induce a non-identity
 automorphism of the residue field $\bR/\gM$ for each maximal ideal $\gM\subset\bR$
 such that $\gM^\si=\gM$ \cite[Theorem 1.3]{chr}. For an action of a group on
 a category (that is, on a principle triple) the notion of separability was introduced
 in \cite{dof}. Obviously, if an action of a group on a bimodule triple is separable,
 so is also its induced action on the corresponding bimodule category. We also
 note that if an action of a group $\sG$ is separable, so is the action of every
 subgroup $\sH\sbe\sG$: if $\tr_\sG\al=1$ and $\be=\sum_{\si\in R}\al^\si$, where
 $R$ is a set of representatives of right cosets $\sH\backslash\sG$, then $\tr_\sH\be=1$.

 Recall that a ring homomorphism $\bA\to\bA'$ is called \emph{separable} if
 the natural homomorphism of $\bA'$-bimodules $\bA'\*_\bA\bA'\to\bA'$ sending
 $a\*b$ to $ab$ \emph{splits}, i.e. there is an element $\sum_ib_i\*c_i$ in
 $\bA'\*_\bA\bA'$ such that $\sum_ib_ic_i=1$ and $\sum_iab_i\*c_i=\sum_ib_ic_ia$
 for all $a\in\bA'$.

 \begin{lemma}\label{l32}
 An action of a finite group $\sG$ on a triple $\dT$ is separable \iff so is the ring
 homomorphism $\kZ\to\kZ\sG$, where $\kZ=\kZ(\dT)$.
  \end{lemma}
 \begin{proof}
 Suppose that the action is separable, $\al=\set{\al_X}$ is such an element of the center
 that $\tr\al=1$. Let $t=\sum_\si\al^\si[\si]\*[\si^{-1}]\in\kZ\sG\*_\kZ\kZ\sG$.
 Then $\sum_\si\al^\si[\si][\si^{-1}]=\tr\al=1$ and, for any $\be\in\kZ,\ \tau\in\sG$,
 \begin{align*}
 \be[\tau]\cdot t&=\sum_\si\be\al^{\tau\si}[\tau\si]\*[\si^{-1}]=
 \sum_\si\al^{\tau\si}\be[\tau\si]\*[\si^{-1}]=\\ &=\sum_\si \al^\si\be[\si]\*[\si^{-1}\tau]=
 \sum_\si\al^\si[\si]\*[\si^{-1}]\be[\tau]=t\cdot \be[\tau],
 \end{align*}
 so the homomorphism $\kZ\to\kZ\sG$ is separable.

 Now let the homomorphism $\kZ\to\kZ\sG$ be separable. Note that every element from
 $\kZ\sG\*_\kZ\kZ\sG$ is of the form $\sum_{\si,\tau}z_{\si,\tau}[\si]\*[\tau]$
 for some $z_{\si,\tau}\in\kZ$. Hence there are elements $z_{\si,\tau}$ such that
 $\sum_{\si,\tau}z_{\si,\tau}[\si\tau]=\sum_\tau\big(\sum_\si z_{\si,\si^{-1}\tau}\big)[\tau]=1$,
 i.e. $\sum_\si z_{\si,\si^{-1}}=1$, and $\sum_\si z_{\si,\si^{-1}\tau}=0$ if $\tau\ne1$,
 moreover, for every $\rho\in\sG$ we have:
 \begin{align*}
  & [\rho]\big(\sum_{\si,\tau} z_{\si,\tau}[\si]\*[\tau]\big)=\sum_{\si,\tau}z^\rho_{\si,\tau}[\rho\si] \*[\tau]
 =\sum_{\si,\tau}z^\rho_{\rho^{-1}\si,\tau}[\si]\*[\tau] = \\
 &=\big(\sum_{\si,\tau}z_{\si,\tau}[\si]\*[\tau]\big)[\rho]=\sum_{\si,\tau}z_{\si,\tau}[\si]\*[\tau\rho]=
 \sum_{\si,\tau}z_{\si,\tau\rho^{-1}}[\si]\*[\tau]. 
 \end{align*}
 Thus $z^\rho_{\rho^{-1}\si,\tau}=z_{\si,\tau\rho^{-1}}$ for $\rho,\si,\tau$. Especially, for
 $\si=\rho,\ \tau=1$ we get $z_{\si,\si^{-1}}=z^\si_{1,1}$. Therefore, $\tr z_{1,1}=1$ and
 the action is separable. 
 \end{proof}

 \begin{corol}\label{c33}
 If an action of a group $\sG$ on a triple $\dT=(\kA,\kB,\dd)$ is separable, so is also
 the embedding functor $\kA\to\kA\sG$, i.e. the homomorphism of $\kA\sG$-bimodules
 $\phi:\kA\sG\*_\kA\kA\sG\to\kA\sG$ splits, or, the same, for every object $X\in\ob\kA$ there is
 an element $t_X\in(\kA\sG\*_\kA\kA\sG)(X,X)$ such that $\phi(t_X)=1_X$ and $at_X=t_Ya$
 for each $a\in\kA\sG(X,Y)$. In particular, the action of a group $\sG$ on a category $\kA$
 is separable \iff so is the embedding functor $\kA\to\kA\sG$.
 \end{corol}

 \begin{theorem}\label{t34}
 If an action of a finite group $\sG$ on a bimodule triple $\dT=(\kA,\kB,\dd)$ is
 separable, the functor $\Phi:\el(\dT)\sG\to\el(\dT\sG)$ induces an equivalence
 of the categories $\add\el(\dT)\sG\to\el(\dT\sG)$. 
 \end{theorem}
 \begin{proof}
 First we prove a lemma about fully additive categories.

 \begin{lemma}\label{l35}
 Let $\kC$ be a fully additive category, $F:\kC\to\kC'$ be a fully faithful
 functor. $F$ is an equivalence of categories \iff every object $X'\in\kC'$ 
 is isomorphic to a direct summand of an object of the form $FY$, where
 $Y\in\ob\kC$.
 \end{lemma}
 \begin{proof}
 The necessity of this condition is obvious, so we only have to prove the sufficiency.
 If $X'$ is a direct summand of $FY$, there are morphisms $\io':X'\to FY$ and $\pi':FY\to X'$
 such that $\pi'\io'=1_{X'}$. Then $e'=\io'\pi'$ is an idempotent endomorphism of the
 object $FY$. Since the functor $F$ is fully faithful, $e'=Fe$ for an idempotent
 endomorphism $e:Y\to Y$. Since the category $\kC$ is fully additive, there are an
 object $X$ and morphisms $\io:X\to Y$ and $\pi:Y\to X$ such that $e=\io\pi$ and
 $\pi\io=1_X$. Then $(F\io)(F\pi)=e'$ and $(F\pi)(F\io)=1_{FX}$. Let
 $u=\pi'F(\io),\ v=(F\pi)\io'$; then we immediately get that $uv=1_{X'}$ and $vu=1_{FX}$,
 i.e. $X'\simeq FX$, the functor $F$ is also dense, so it is an equivalence of categories.
 \end{proof}

 We prove now that every object $\xi$ of the category $\el(\dT\sG)$ is
 isomorphic to a direct summand of $\Phi\Psi\xi$. Since $\Phi$ is fully faithful
 (Proposition \ref{p21}), Theorem \ref{t34} follows then from Lemma \ref{l35}. Let
 $\xi=\sum_\si x_\si[\si]\in\kB\sG(X,X)$, where $x_\si\in\kB(X^\si,X)$.
 Then $\Psi\xi=\ti\xi\in\kB(\ti X,\ti X)$, where $\ti X=\bop_\si X^\si$ and
 $\ti\xi_{\si,\tau}=x^\si_{\si^{-1}\tau}\la_{\si,\si^{-1}\tau}$, and $\Phi\Psi\xi=\ti\xi[1]$.
 Choose an element $\al\in\kZ(\dT)$ such that $\tr\al=1$. Consider the morphism
 $\pi:\ti X\to X$ such that its $\si$-component equals
 $\pi_\si=\la^{-1}_{\si^{-1},\si}[\si^{-1}]:X^\si\to X$. Then the $\si$-component
 of the element $\xi\pi$ equals
 \[
  \sum_\rho x_\rho(\la^\rho_{\si^{-1},\si})^{-1}\la_{\rho,\si^{-1}}[\rho\si^{-1}]
 = \sum_\rho x_\rho\la^{-1}_{\rho\si^{-1},\si}[\rho\si^{-1}]
 \]
 (we use the relation \eqref{e21} for the triple $\rho,\si^{-1},\si$), while
 the $\si$-component of the element $\pi\ti\xi[1]$ equals
 \begin{align*}
 &\sum_\rho \la^{-1}_{\rho^{-1},\rho}(x^\rho_{\rho^{-1}\si})^{\rho^{-1}}
 \la_{\rho,\rho^{-1}\si}^{\rho^{-1}}[\rho^{-1}]=
 \sum_\rho x_{\rho^{-1}\si}\la^{-1}_{\rho,\rho^{-1}}\la^{\rho^{-1}}_{\rho,\rho^{-1}\si}[\rho^{-1}]
 =\\&\ =\sum_{\rho} x_{\rho^{-1}\si}\la^{-1}_{\rho^{-1},\si}[\rho^{-1}] =
 \sum_\rho x_\rho\la^{-1}_{\rho\si^{-1},\si}[\rho\si^{-1}].
 \end{align*}
 Here we used first the relation \eqref{e21} for the triple $\rho^{-1},\rho,\rho^{-1}\si$ and
 then replaced $\rho$ by $\si\rho^{-1}$. So $\xi\pi=\pi\ti\xi[1]$ and, since $\dd\pi=0$,
 $\pi$ is a morphism $\ti\xi[1]\to\xi$. Now consider the morphism $\io:X\to\ti X$ such
 that its $\si$-component equals $\al_{X^\si}[\si]$. The $\si$-component of the
 element $\io\xi$ equals
 \[
  \sum_\rho \al_{X^\si} x^\si_\rho \la_{\si,\rho}[\si\rho]=
 \sum_\rho \al_{X^\si} x^\si_{\si^{-1}\rho}\la_{\si,\si^{-1}\rho}[\rho],  
 \]
 and the $\si$-component of the element $\ti\xi[1]\io$ equals
 \[
  \sum_\rho x^\si_{\si^{-1}\rho}\la_{\si,\si^{-1}\rho}\al_{X^\rho}[\rho]=
   \sum_\rho\al_{X^\si}x^\si_{\si^{-1}\rho}\la_{\si,\si^{-1}\rho} [\rho],
 \]
 since $\al\in\kZ(\dT)$. Therefore, $\ti\xi[1]\io=\io\xi$, thus $\io$ is a
 morphism $\xi\to\ti\xi[1]$. But $\pi\io=\sum_\si\la^{-1}_{\si^{-1},\si}
 \al_{X^\si}^{\si^{-1}}\la_{\si^{-1},\si}=(\tr\al)_X=1_X=1_\xi$, which just 
 means that the element $\xi$ is a direct summand of the element $\ti\xi[1]$.
 \end{proof}

 One can get more information if the group $\sG$ is finite abelian and the ring
 $\bK$ is a field containing a \emph{primitive $n$-th root of unit}, where
 $n=\#(\sG)$, i.e. such an element $\ze$ that $\ze^n=1$ and $\ze^k\ne1$
 for $0<k<n$. Then certainly $\chr\bK\nmid n$, so any action of the group $\sG$ 
 on a bimodule triple $\dT=(\kA,\kB,\dd)$ is separable.
 
 Let $\hG$ be the \emph{group of characters} of the group $\sG$, i.e. the group of
 its homomorphisms to the multiplicative group $\bK^\xx$ of the field $\bK$. 
 This group acts on the triple $\dT\sG$ (with the trivial system of factors) by the rules:
 \begin{align*}
  & X^\chi=X \ \text{ for every }\ X\in\ob\kA,\\ 
  &\big(\sum_\si x_\si[\si]\big)^\chi=\sum_\si \chi(\si)x_\si[\si],
 \end{align*}
 where $\chi\in\hG$ and $\sum_\si x_\si[\si]$ is a morphism from $\kA\sG$ or
 an element from $\kB\sG$. Recall that also $\#(\hG)=n$, so this action is separable
 as well. We denote by $\chi_0$ the \emph{unit character}, i.e. such that $\chi_0(\si)=1$ 
 for all $\si\in\sG$. By definition, morphisms from $\kA\sG\hG$ and elements of $\kB\sG\hG$
 are of the form $\sum_{\si,\chi} x_{\si,\chi}[\si][\chi]$. We write $[\chi]$ instead of
 $[1][\chi]$ and $\si$ instead of $[\si][\chi_0]$. In particular an element
 $x[1][\chi_0]$ is denoted by $x$. 

 \begin{theorem}\label{t36}
 The bimodule triples $\add\dT$ and $\add\dT\sG\hG$ are equivalent.
 \end{theorem}
 \begin{proof}
 Consider the elements $\ds e_\si=\frac1n\sum_\chi\chi(\si)[\chi]$ from the endomorphism
 ring $\kA\sG\hG(X,X)$. The formulae of orthogonality for characters \cite[Theorem 3.5]{dk} 
 immediately imply that $e_\si$ are mutually orthogonal idempotents and $\sum_\si e_\si=1$.
 Moreover, $e_\si[\tau]=[\tau]e_{\si\tau}$, so all these idempotents are conjugate, thus define
 isomorphic direct summands $X_\si$ of the object $X$ in  the category $\add\kA\sG\hG$,
 and $X=\bop_\si X_\si$. We define the bifunctor $\Th:\add\dT\to\add\dT\sG\hG$ setting
 $\Th X=X_1$ and $\Th x=xe_1=e_1x$, where $x$ is a morphism $X\to Y$ or an element
 from $\kB(X,Y)$. Obviously, the functor $\Th_0:\add\kA\to\add\kA\sG\hG$ satisfies the
 conditions of Lemma \ref{l35}, so it defines an equivalence of categories. Since every
 map $\Th_1(X,Y)$ is also bijective, the bifunctor $\Th$ is an equivalence by Lemma \ref{l11}.
 \end{proof}

 \begin{corol}\label{c37}
 The categories  $\el(\dT)$ and $\add\el(\dT)\sG\hG$ are equivalent.
 \end{corol}
 \begin{proof}
  Indeed, $\add\el(\dT)\sG\hG\simeq\el(\dT\sG\hG)$ by Theorem \ref{t34}.
 \end{proof}

 \section{Radical and decomposition}
 \label{s4}

 In this section we suppose that the ring $\bK$ is \emph{noetherian, local and henselian}
 \cite{bo} (for instance, \emph{complete}). We denote by $\gM$ its maximal ideal and by
 $\gK=\bK/\gM$ its residue field. We call a $\bK$-category $\kA$ \emph{piecewise finite}
 if all $\bK$-modules $\kA(X,Y)$ are finitely generated. Then its additive hull $\add\kA$
 is piecewise finite as well. Moreover, each endomorphism ring $A=\kA(X,X)$ is
 \emph{semiperfect}, i.e. possesses a unit decomposition $1=\sum_{i=1}^ne_i$, 
 where $e_i$ are mutually orthogonal idempotents and all rings $e_iAe_i$ are local.
 Hence the category $\add\kA$ is \emph{local}, i.e. every object in it decomposes into
 a finite direct sum of objects with local endomorphism rings. Therefore this category is
 a \emph{Krull--Schmidt category}, i.e. every object $X$ in it decomposes into a finite
 direct sum of indecomposables: $X=\bop_{i=1}^mX_i$ and such a decomposition is
 unique, i.e. if also $X=\bop_{i=1}^nX'_i$, where all $X'_i$ are indecomposable, then
 $m=n$ and there is a permutation $\eps$ of the set $\set{1,2,\dots,m}$ such that 
 $X_i\simeq X'_{\eps i}$ for all $i$ \cite[Theorem I.3.6]{ba}.
 Recall that the \emph{radical} of a local category $\kA$ is the ideal $\rad\kA$ 
 consisting of all such morphisms $a:X\to Y$ that all components of $a$ with respect
 to some (then any) decompositions of $X$ and $Y$ into a direct sum of indecomposables
 are non-invertible. We denote $\lA=\kA/\rad\kA$. In particular, $\rad\kA(X,X)$ is the
 radical of the ring $\kA(X,X)$ and $\lA(X,X)$ is a semisimple artinian ring \cite{gr}.
 In the case of a piecewise finite category always $\rad\kA\spe\gM\kA$, in particular,
 $\lA(X,X)$ is a finite dimensional $\gK$-algebra. The category $\lA$ is \emph{semisimple},
 i.e. every object in it decomposes into a finite direct sum of indecomposables and
  $\lA(X,Y)=0$ if $X$ and $Y$ are non-isomorphic indecomposables, while $\lA(X,X)$ is
 a skewfield for every indecomposable object $X$. (Note that an object $X$ is indecomposable
 in the category $\kA$ \iff it is so in the category $\lA$). Moreover, $\rad\kA$ is the biggest
 among the $\kI\subset\kA$ such that the factor-category $\kA/\kI$ is semisimple.

 If a finite group $\sG$ acts on a piecewise finite category $\kA$ with a system of factors
 $\la$, the category $\kA\sG$ is piecewise finite as well. Moreover, the radical is a
 $\sG$-invariant ideal, i.e. $(\rad\kA)^\si=\rad\kA$ for all $\si\in\sG$, and the ideal
 $(\rad\kA)\sG$ is contained in the radical of the category $\kA\sG$. 

 \begin{prop}\label{p41}
 If the action of a group $\sG$ on a category $\kA$ is separable, so is also its induced
 action on the category $\lA\sG$. In this case $\rad(\kA\sG)=(\rad\kA)\sG$ and the category
  $\lA\sG$ is semisimple.
 \end{prop}
 \pev. 

 From now on, we suppose that $\kA$ is a piecewise finite local $\bK$-category,
 $\kR=\rad\kA$, $X\in\ob\kA$ is an indecomposable object from $\kA$,
 $\bA=\kA(X,X)$ and $\sG$ is a finite group acting on $\kA$ with a system of factors
 $\la$ so that its action is separable. We are interested in the decomposition of the
 object $X$ in the category $\kA\sG$ into a direct sum of indecomposables, especially,
 the number $\nu_\sG(X)$ of non-isomorphic summands in such a decomposition. 
 Recall that such decomposition comes from a decomposition of the ring
 $\kA\sG(X,X)$ or, equivalently, of the ring $\lA\sG(X,X)$ into a direct sum of
 indecomposable modules. 

 \begin{prop}\label{p42}
 Let $\sH=\setsuch{\si\in\sG}{X^\si\simeq X}$. Then 
 \[
   \kA\sG(X,X)/\kR\sG(X,X) \simeq\kA\sH(X,X)/\kR\sH(X,X), 
 \]
 in particular, $\nu_\sG(X)=\nu_\sH(X)$.
 \end{prop}
 \pev{, since $a_\si\in\kR$ for every morphism $a_\si:X^\si\to X$ if 
 $\si\notin\sH$.}

 \begin{corol}\label{c42}
 If $X^\si\not\simeq X$ for all $\si\in\sG$, the object $X$ remains indecomposable
 in the category $\kA\sG$.
 \end{corol}
 
 Therefore, dealing with the decomposition of $X$, we can only consider the action of
 the subgroup $\sH$. For every $\si\in\sH$ we fix an isomorphism $\phi_\si:X^\si\to X$
 and consider the action $T'$ of the group $\sH$ on the ring $\bA$ given by the rule
 $T'_\si(a)=\phi_\si a^\si\phi_\si^{-1}$. One easily verifies that the elements
 $\la'_{\si,\tau}=\phi_\si\phi^\si_\tau\la_{\si,\tau}\phi_{\si\tau}^{-1}$ form a system
 of factors for this action, moreover, the map $a[\si]\mapsto a\phi_\si[\si]$ 
 establishes an isomorphism $\bA(\sH,T',\la')\simeq\kA\sH(X,X)$. Thus, in what follows,
 we investigate the algebras $\bA(\sH,T',\la')$ and $\bD(\sH,T',\ola)$, where
 $\bD=\bA/\rad\bA$ and $\ola_{\si,\tau}$ denotes the image of $\la'_{\si,\tau}$
 in the skewfield $\bD$. The latter factor-ring is finite dimensional skewfield (division
 algebra) over the field $\gK$. We denote by $\bF$ the center if this algebra (it is a field).
 Let $\sN$ be the subgroup of $\sH$ consisting of all elements $\si$ such that the
 automorphism $T'_\si$ induces an inner automorphism of the skewfield $\bD$, or, 
 equivalently, the identity automorphism of the field $\bF$ \cite[Corollary IV.4.3]{dk}. 
 It is a normal subgroup in $\sH$. For every element $\rho\in\sN$ we choose an element
 $d_\rho\in\bD$ such that $T'_\rho(a)=d_\rho ad_\rho^{-1}$ for all $a\in\bD$. We also
 choose a set $\kS$ of representatives of cosets $\sH/\sN$ and, for every $\si\in\sH$,
 denote by $\lsi$ the element from $\kS$ such that $\si\sN=\lsi\sN$, and by $\rho(\si)$
 the element from $\sN$ such that $\si=\rho(\si)\lsi$. Now we set
 $D_\si(a)=d_{\rho(\si)}^{-1}T'_{\si}(a)d_{\rho(\si)}$. An immediate verification shows
 that we get in this way an  action of the group $\sH$ on the skewfield $\bD$ with the
 system of factors
 $\mu_{\si,\tau}=d_{\rho(\si)}^{-1}(d_{\rho(\tau)}^\si)^{-1} \ola_{\si,\tau}d_{\rho(\si\tau)}$
 and, besides, the map $[\si]\mapsto d_{\rho(\si)}[\si]$ induces an isomorphism
 $\bD(\sH,T',\ola)\simeq\bD(\sH,D,\mu)$. Note that now 
 \[
  \sN=\setsuch{\si\in\sH}{D_\si=\id}=\setsuch{\si\in\sH}{D_\si|_\bF=\id}. 
 \]
 Moreover, one easily sees that $\mu_{\si,\tau}\in\bF$ if $\si,\tau\in\sH$. 

 Further on we denote $\bD\sH=\bD(\sH,D,\mu)$. The number of non-isomorphic
 indecomposable summands in the decomposition of $\bD\sH$ equals the number of
 simple components of this algebra \cite[Theorem II.6.2]{dk}, or, the same, the number
 of simple components of its center.
 
 \begin{prop}\label{p43}
 The center of the algebra $\bD\sH$ coincides with the set
 \begin{align*}
  &(\bF\sN)^\sH=\setsuch{\al\in\bF\sH}{\LA\tau\, [\tau]\al=\al[\tau]} =\\
  &\ = \Big\{\sum_{\si\in\sN}a_\si[\si]\,\Big|\, \LA\si \big(a_\si\in\bF\, \&\, \LA\tau(\tau\in\sH \Arr 
  a_\si^\tau\mu_{\tau,\si}=a_{\tau\si\tau^{-1}}\mu_{\tau\si\tau^{-1},\tau})\big)\Big\}.  
 \end{align*}
 Especially, if $\sN=\set1$, then $\bD\sH$ is a central simple algebra over the field of
 invariants $\bF^\sH$, hence, $\nu_\sG(X)=1$.%
 \footnote{\ The last statement is well-known, see \cite[Theorem 4.50]{jac}.}
 \end{prop}
 \begin{proof}
 If an element $\al=\sum_\si a_\si[\si]$ belongs to the center of $\bD\sH$, then 
 $\sum_\si ba_\si[\si]=\sum_\si a_\si[\si]b=\sum_\si a_\si b^\si[\si]$, so if
 $a_\si\ne0$, then $b^\si=a_\si^{-1} ba_\si$, hence, $\si\in\sN,\ b^\si=b$ and $a_\si\in\bF$.
 Finally, the equalities $[\tau] \al=\sum_\si a^\tau_\si \mu_{\tau,\si}[\tau\si]=\al[\tau]=
 \sum_\si a_\si\mu_{\si,\tau}[\si\tau]=\sum_\si a_{\tau\si\tau^{-1}}
 \mu_{\tau\si\tau^{-1},\tau}[\tau\si]$ complete the proof. 
 \end{proof}
 
 \begin{corol}\label{c44}
  If $\bF=\gK$ (for instance, the residue field $\gK$ is algebraically closed)
 and the group $\sH$ is abelian, the center of the algebra $\bD\sH$ coincides with
 $\gK\sH_0$, where $\sH_0$ is the subgroup of $\sH$ consisting of all elements
 $\si$ such that $\mu_{\si,\tau}=\mu_{\tau,\si}$ for all $\tau\in\sH$. In particular,
 $\nu_\sG(X)=\#(\sH_0)$.
 \end{corol}
 \begin{proof}
 In this case $\sN=\sH$, so the center of $\bD\sH$ coincides with $\gK\sH_0$
 (one easily checks that $\sH_0$ is indeed a subgroup). Since the latter algebra
 is commutative  and semisimple, it is isomorphic to $\gK^m$, where $m=\#(\sH_0)$,
 therefore, the number of its simple components equals $m$.
 \end{proof}
 
 \begin{corol}\label{c45}
  If $\bF=\gK$ and the group $\sH$ is cyclic, the center of the algebra $\bD\sH$
 coincides with $\gK\sH$ and $\nu_\sG(X)=\#(\sH)$.
 \end{corol}
 \begin{proof}
 Actually, in this case it is well-known that $\mu_{\si,\tau}=\mu_{\tau,\si}$
 for all $\si,\tau\in\sH$. 
 \end{proof}

 Note that all these corollaries hold if the group $\sG$ itself is abelian or cyclic. 

 If $\bK$-category $\kA$ is piecewise finite, so is every bimodule category $\el(\dT)$ 
 as well, where $\dT=(\kA,\kB,\dd)$. If a group $\sG$ acts separably on the triple $\dT$,
 it acts separably on the category $\el(\dT)$ as well, and, according to Theorem \ref{t34},
 $\add\el(\dT)\sG\simeq\el(\dT\sG)$, this equivalence being induced by the functor 
 $\Phi:x\mapsto x[1]$. Therefore, all the results above can be applied to the study of the
 decomposition of an element $x[1]$ in the category $\el(\dT\sG)$. We only quote explicitly
 the reformulations of Corollaries \ref{c44} and \ref{c45} for this case. 

 \begin{corol}\label{c46}
 Let the residue field $\gK$ be algebraically closed and the group $\sH=\setsuch{\si}{x^\si\simeq x}$
 be abelian. Choose isomorphisms $\phi_\si:x^\si\to x$ for every element $\si\in\sH$ and denote
 by $\mu_{\si,\tau}$ the image of a morphism $\phi_\si\phi_\tau\la_{\si,\tau}\phi_{\si\tau}^{-1}$
 in $\gK\simeq\Hom_\dT(x,x)/\rad_\dT(x,x)$. Then the number of non-isomorphic indecomposable
 direct summands in the decomposition of the object $x[1]$ in the category $\el(\dT\sG)$ 
 equals the order of the group $\sH_0=\setsuch{\si}{\LA\tau\,\mu_{\si,\tau}=\mu_{\tau,\si}}$. 
 Especially, if the group $\sH$ is cyclic, this number equals the order of $\sH$. 
 \end{corol}

 \begin{rmk} 
 It is evident that all these statements also hold if separable is the action of the group $\sH$
 on the skewfield $\bD$, or, equivalently, on its center $\bF$. It is known \cite[Section 4.18]{jac}
 that one only has to verify that separable is the action of the subgroup $\sN$, i.e. that
 $\chr\gK\nmid\#(\sN)$, since the action of $\sN$ on $\bF$ is trivial.
 \end{rmk}
    
 Proposition \ref{p41} evidently implies some more corollaries concerning the structure
 of the radical of the category $\kA\sG$ (for instance, bimodule category $\el(\dT\sG)\hskip1pt$).

 \begin{corol}\label{c47}
 Let the action of the group $\sG$ is separable. If a set of morphisms $\set{a_i}$ is a set of
 generators of the $\kA$-module $(\rad\kA)(X,\_)$ (or $\kA\op$-module $(\rad\kA)(\_\,,X)\,$), 
 its image $\set{a_i[1]}$ in $\kA\sG$ is a set of generators of the $\kA\sG$-module
 $(\rad\kA\sG)(X,\_)$ (respectively, $\kA\op$-module $(\rad\kA\sG)(\_\,,X)\hskip1pt$).
 \end{corol} 

 We call a morphism $a:Y\to X$ \emph{left almost split} (respectively, \emph{right almost split})
 if it generates the $\kA$-module $(\rad\kA)(\_\,,X)$ (respectively, $\kA\op$-module $(\rad\kA)(Y,\_)\,$),
 and an equality $a=bf$ implies that the morphism $f$ is left invertible, or, the same, is a
 split epimorphism (respectively, the equality $a=fb$ implies that $g$ is right invertible, or, the same, 
 is a split monomorphism).%
 \footnote{\ In the book \cite{ars} one only uses these notions in  the case when $X$ (respectively,
 $Y$) is indecomposable. However, one can easily see that a left (right) almost split morphism in
 our sense is just a direct sum of those in  the sense of \cite{ars}. The same also concerns the
 notion of the \emph{almost split sequences} used below.}

 \begin{corol}\label{c48}
 Let the action of $\sG$ is separable. If a morphism $a:Y\to X$ is left (right)
 almost split, so is $a[1]$ as well.
 \end{corol}

 A sequence $X\xarr{a} Y\xarr{b} X'$ is called \emph{almost split} if the morphism
 $a$ is left almost split, the morphism $b$ is right almost split and, besides,
 $a=\Ker b$ and $b=\cok a$, i.e.,  for every object $Z$, the induced sequences 
 of groups
 \begin{align*}
  0&\to \kA(Z,X)\to \kA(Z,Y)\to \kA(Z,X'),\\
  0&\to \kA(X',Z)\to \kA(Y,Z)\to \kA(X,Z)
 \end{align*}
  are exact.

 \begin{corol}\label{c49}
  Let the action of $\sG$ is separable. If a sequence $X\xarr{a} Y\xarr{b} X'$ is almost split
 in the category $\kA$, the sequence $X\xarr{a[1]} Y\xarr{b[1]} X'$ is almost split in the
 category $\kA\sG$.
 \end{corol}
 
 Since, under the separability condition, every object from $\add\kA\sG$ is a direct summand of
 an object that has come from the category $\kA$, Corollaries \ref{c48} and \ref{c49} describe
 almost split morphisms and sequences in the category $\add\kA\sG$ as soon as they are known
 in the category $\kA$. In particular, these results can be applied to the bimodule categories
 $\el(\dT\sG)$ due to Theorem \ref{t34}.

 \end{document}